\documentclass[12pt,a4paper,reqno]{amsart}
\pdfoutput=1
\usepackage[hidelinks,pdfstartview=FitH]{hyperref} 
\usepackage{amssymb,tikz}
\usepackage[all,cmtip]{xy}
\usepackage{verbatim}
\usepackage[USenglish]{babel}
\usepackage[top=3cm,right=2.5cm,
left=2.5cm,bottom=3cm,marginparsep=15pt,marginparwidth=2cm,footskip=20pt]{geometry}

\usepackage{color} 

\allowdisplaybreaks[4]

\newtheorem{prop}{Proposition}[section]

\newtheorem{thm}[prop]{Theorem}

\newtheorem*{thm*}{Theorem}

\numberwithin{equation}{section}

\newcommand{\Z}{\mathbb{Z}}

\newcommand{\C}{\mathbb{C}}

\newcommand{\cL}{\mathcal{L}}
\newcommand{\cE}{\mathcal{E}}

\newcommand{\bS}{S}

\newcommand{\Cs}{C*-}

\newcommand{\id}{\mathrm{id}}

\newcommand{\inn}[1]{\langle #1 \rangle}

\renewcommand{\[}{\begin{equation}}
\renewcommand{\]}{\end{equation}}

\title[Associated vector bundles over quantum projective spaces]{\vspace*{-23mm}
\mbox{Associated noncommutative vector bundles over}\\
\hspace*{-8mm}\mbox{the Vaksman--Soibelman quantum complex projective spaces}} 

\author[F.~Arici]{Francesca Arici} 
\address[F.~Arici]{Max-Planck-Institut f\"ur Mathematik in den Naturwissenschaften, Inselstr.~22, 04103 Leipzig, Leipzig, Germany.}
\email{francesca.arici@mis.mpg.de}
 
\author[P. M.~Hajac]{Piotr M.~Hajac}
\address[P. M.~Hajac]{Instytut Matematyczny, Polska Akademia Nauk, ul. \'Sniadeckich 8, Warszawa, 00-656 Poland}
\email{pmh@impan.pl }

\author[M.~Tobolski]{Mariusz Tobolski}
\address[M.~Tobolski]{Instytut Matematyczny, Polska Akademia Nauk, ul. \'Sniadeckich 8, Warszawa, 00-656 Poland}
\email{mtobolski@impan.pl }

\subjclass[2010]{46L80, 46L85, 58B32, 57R22}

\keywords{K-theory, noncommutative vector bundle, compact quantum group, Peter--Weyl decomposition, principal comodule algebra}

\begin{document}
\baselineskip15pt
\parskip=0.5\baselineskip

\begin{abstract}
By a diagonal embedding of $U(1)$ in $SU_q(m)$, 
we prolongate the diagonal circle action on 
the Vaksman--Soibelman quantum sphere~$\bS^{2n+1}_q$ to the $SU_q(m)$-action
on the prolongated bundle.
Then we prove that
the noncommutative vector bundles associated via the fundamental representation of~$SU_q(m)$,
for $m\in\{2,\ldots,n\}$,
yield generators of the even K-theory group of the C*-algebra 
 of the Vaksman--Soibelman quantum complex projective space
 $\mathbb{C}{\rm P}^n_q$.
\end{abstract}
\maketitle

\section{Introduction}
The K-theory of complex projective spaces was unraveled by Atiyah and Todd in \cite[Prop. 2.3, 3.1 and 3.3]{AT60} 
(cf.~\cite[Thm.~7.2]{A62}~and~\cite[Cor. IV.2.8]{Ka78}).
Denoting by $\mathrm{L}^n_1$ the dual tautological line bundle over $\mathbb{C}{\rm P}^n$ and setting
$$
\mathrm{t}:=[\mathbb{C}{\rm P}^n\!\times\!\mathbb{C}]-[\mathrm{L}^n_1]\in K^0(\mathbb{C}{\rm P}^n),
$$
one obtains
\[\label{atiyah}
K^0(\mathbb{C}{\rm P}^n) =\Z[\mathrm{t}] / \mathrm{t}^{n+1} \, .
\]
Recently, this result was extended to the Vaksman--Soibelman quantum complex projective spaces 
\cite[Prop. 3.3 and 3.4]{ABL15}:
\[\label{truncated}
K_0(C(\mathbb{C}{\rm P}^n_q)) = \Z[t] / t^{n+1} \, .
\]
Here $0<q<1$, $t := [1] - [L^n_1]$, and $L^n_1$ is the section module of the noncommutative dual tautological line bundle
over~$\mathbb{C}{\rm P}^n_q$.

The K-theory of both the Vaksman--Soibelman and the multipullback quantum complex projective \emph{planes}
\cite{HKZ12,Ru12,HR17} was thoroughly analysed in~\cite{FHMZ17}. In this special $n=2$ case, the following elements
\[\label{set}
[1],\quad [L^2_1]-[1],\quad [L^2_{-1}\oplus L^2_1]-2[1]
\]
form a basis of~$K_0(C(\mathbb{C}{\rm P}^2_q))$. 
Here $L^2_{-1}$ is the section module of the noncommutative  tautological line bundle
over~$\mathbb{C}{\rm P}^2_q$.

A particular feature of the module $L^2_{-1}\oplus L^2_1$ established 
in \cite{FHMZ17} is
that it can be obtained as a Milnor module \cite[Thm.~2.1]{Mi71} constructed from the fundamental representation of~$SU_q(2)$.
This was achieved by taking $U(1)$ as a subgroup of~$SU_q(2)$, and then realising $L^2_{-1}\oplus L^2_1$ as
the section module of the noncommutative \emph{vector} bundle associated to the prolongation
$\bS^{5}_q\times_{U(1)}SU_q(2)$ via the fundamental representation of~$SU_q(2)$.
The upshot of having $[L^2_{-1}\oplus L^2_1]-2[1]$ as the image of a \Cs algebraic Milnor connecting homomorphism \cite[Sec. 0.4]{HRZ13} is that
it allows one to prove that the set \eqref{set} is a basis without using the index pairing over~$\mathbb{C}{\rm P}^2_q$. 

The goal of this paper is to show that the above construction works in
any dimension. More precisely, except for the classes $[1]$ and $[L^n_1]-[1]$, we prove that all other generators 
of $K_0(C(\mathbb{C}{\rm P}^n_q))$ come from the section module of the noncommutative vector bundle associated to the prolongation
$\bS^{2n+1}_q\times_{U(1)}SU_q(m)$, for  $m\in\{2,\ldots,n\}$, via the fundamental representation of~$SU_q(m)$.

Finally, since there is the  exact sequence \cite{ADHT18}
$$
0\longrightarrow K_1(C(S^{2n-1}_q))\stackrel{\partial_{10}}{\longrightarrow} K_0(C(\mathbb{C}{\rm P}^n_q))
\stackrel{}{\longrightarrow} K_0(C(\C{\rm P}^{n-1}_q))\longrightarrow 0\,, 
$$
a generator of $K_0(C(\mathbb{C}{\rm P}^n_q))$ that does not come from $\C{\rm P}^{n-1}_q$
can always be constructed via the Milnor connecting homomorphism
$\partial_{10}$
from a generator of $K_1(C(S^{2n-1}_q))$. Combining this with the fact that $S^{2n-1}_q$ is a homogeneous space
of~$SU_q(n)$, we arrive at the following:

\noindent\textbf{Question}: Can a generator of the $K_1(C(S^{2n-1}_q))\cong\mathbb{Z}$ be always expressed in
terms of representations of~$SU_q(n)$?

\noindent
This question has positive answer for $q=1$ by the work of Harris \cite{Ha68} based on the work of Atiyah and Hodgkin
\cite{At65,Ho67}. For $0<q<1$ and $n=2$, we just take the fundamental representation of~$SU_q(2)$.

\section{Preliminaries}

The Vaksman--Soibelman odd quantum spheres \cite{VS91} are defined as quantum homogenous spaces for Woronowicz's quantum special unitary groups \cite{Wor88}:
 $$
 C(S^{2n+1}_q):=C({SU}_q(n+1))^{{SU}_q(n)}.
 $$
Here $0<q\leq 1$, and we use the notation $A^\mathbb{G}$ for the fixed-point subalgebra of a C*-algebra $A$ under
an action of a compact quantum group~$\mathbb{G}$. (Note that the $q=1$ case recovers the classical situation.)
One can show that $C(S^{2n+1}_q)$ is the universal \Cs algebra given by the following generators and relations:
\begin{gather*}
z_iz_j = qz_jz_i \quad \mbox{for $i<j$}, \qquad z_iz^*_j = qz_j^*z_i \quad \mbox{for $i\neq j$},\\
z_iz_i^* = z_i^*z_i + (q^{-2}-1)\sum_{m=i+1}^n z_mz_m^*, \qquad \sum_{m=0}^n z_mz_m^*=1.
\end{gather*}


Much like in the classical case, the Vaksman--Soibelman quantum odd sphere enjoy the diagonal circle action given on generators 
by
$$
(z_0,z_1, \dots, z_n) \mapsto (\lambda z_0, \lambda z_1, \dots, \lambda z_n ), \qquad \lambda \in U(1).
$$
One uses this action to define the quantum complex projective spaces \cite{VS91}
$$
C(\mathbb{C}{\rm P}^n_q): = C(S^{2n+1}_q)^{U(1)}.
$$

Recall that freeness of circle actions can be characterised in terms of their spectral subspaces (see e.g.\ \cite{Pa81}). Given an action $\alpha : U(1) \to \mathrm{Aut}(A)$ on a unital \Cs algebra, for each character $m\in\mathbb{Z}$, one defines the \emph{$m$-th 
spectral subspace} $A_m$ as
\begin{equation*}
A_m := \{a \in A \mid \alpha_\lambda(a) = \lambda^m a\text{ for all }\lambda \in U(1)\}.
\end{equation*}
The subspace $A_0$ agrees with the fixed-point subalgebra $A^{U(1)}$, and 
$A_m A_n\subseteq A_{m+n}$ for all $m,n \in \Z$, turning $A$ into a $\Z$-graded algebra. 
Now, one can say that the action $\alpha$ is \emph{free} if and only if the $\mathbb{Z}$-grading is \emph{strong} \cite{Ul81}, i.e.\ $A_mA_n=A_{m+n}$ for all $m, n \in \Z$. It is straightforward to check that the bimodules $A_m$ are finitely generated projective both as left and right
$A^{U(1)}$-modules. Furthermore, they are invertible and they can be interpreted as modules of sections of \emph{associated 
noncommutative line bundles} (see e.g.\ \cite{AKL16}).

Note that, using the spatial tensor product, the action $\alpha$ can be dualised to the coaction 
$$
\delta: A \longrightarrow A \underset{\mathrm{min}}{\otimes} C(U(1)) = C(U(1),A), \qquad \delta(a)(\lambda) := \alpha_{\lambda}(a) \,.
$$
Denote by $\mathcal{O}(U(1))$ the dense Hopf subalgebra of $C(U(1))$ consisting of Laurent polynomials in one variable. The Peter--Weyl $\mathcal{O}(U(1))$-comodule algebra  $\mathcal{P}_{U(1)} (A)$, defined as the set of all $a \in A$ such that $\delta(a) \in A \otimes \mathcal{O}(U(1))$ (see \cite{BdCH}), is the purely algebraic direct sum of spectral subspaces:
\begin{equation}
\label{eq:PW}
\mathcal{P}_{U(1)} (A) = \bigoplus_{m \in \Z} A_m.
\end{equation}
The $\mathcal{P}_{U(1)} (A)$ comodule algebra is principal in the sense of~\cite{HKMZ11}. 

A centrepiece concept for our paper is the notion of a  prolongation
 of a principal comodule algebra. To define it, we need to recall the construction of
cotensor product.
For a coalgebra $C$,
the cotensor product of a right $C$-comodule $M$ with a coaction~$\rho_R:M\to M\otimes C$ and a left
$C$-comodule $N$ with a coaction $\rho_L:N\to C\otimes N$ is defined as the difference kernel
$$
M\overset{C}{\square}N:=\ker\left(\rho_R\otimes\id-\id\otimes\rho_L:M\otimes N\longrightarrow M\otimes C\otimes N\right).
$$
Given a surjection of Hopf algebras $\pi\colon \mathcal{H}\to\bar{\mathcal{H}}$, 
we can treat $\mathcal{H}$ as a left $\bar{\mathcal{H}}$-comodule
via the coaction $(\pi\otimes\id)\circ\Delta$. For any right $\bar{\mathcal{H}}$-comodule algebra $\mathcal{P}$,
we define its
\emph{prolongation} as the cotensor product
$\mathcal{P}\Box^{\bar{\mathcal{H}}}\mathcal{H}$. It is easy to check that this cotensor product is 
a right $\mathcal{H}$-comodule algebra
for the coaction $\id\otimes\Delta$. It is also straightforward to verify that, if $\mathcal{P}$ is principal, then so
is its prolongation.

Let $\mathcal{P}$ be a principal $\mathcal{H}$-comodule algebra with a
coaction~$\Delta_R$ and let $V$ be a left
$\mathcal{H}$-comodule. 
Much as in the classical case, we can form the associated left module $\mathcal{P}\Box^{\mathcal{H}}V$
over the coaction-invariant subalgebra 
$\mathcal{P}^{\mathrm{co}\,\mathcal{H}}:=\{a\in \mathcal{P}\;|\;\Delta_R(a)=a\otimes 1\}$. We think of this module
as the section module of the \emph{associated noncommutative vector bundle}. If $\mathcal{P}$ is principal and 
$V$ is finite dimensional,
then it is known that the associated module is finitely generated projective.

\section{Line bundles}

As explained in the introduction, the K-theory of the Vaksman--Soibelman quantum complex projective spaces
is determined by section modules of associated noncommutative line bundles. Therefore, we begin our considerations
by putting together some facts involving these noncommutative line bundles.

For the Vaksman--Soibelman quantum sphere, the Peter--Weyl subalgebra \eqref{eq:PW} becomes
$$
\mathcal{P}_{U(1)} (C(S^{2n+1}_q)) = \bigoplus_{m \in \Z} {L}^n_m,
$$
where
$$ L^n_m := \lbrace a \in C(S^{2n+1}_q) \mid \alpha_\lambda(a) = \lambda^m a\text{ for all }\lambda \in U(1)\rbrace.$$

\noindent The $U(1)$-action on $C(S^{2n+1}_q)$ is free by \cite[Cor.~3]{Sz03}, so we can think of $L^n_m$ as modules of sections of 
associated noncommutative line bundles. Note that our sign convention is opposite to that used in 
\cite{ABL15}, i.e.\ $L^n_m=\cL_{-m}$. 

Recall that there exists a $U(1)$-equivariant  $*$-homomorphism 
$\varphi: C(S^{2n+1}_q)\rightarrow C(S^3_q)$ given on standard generators by
\begin{equation*}
\varphi(z_i) =\begin{cases}
    z_0 &\text{ if $i = 0$} \\
    z_1 &\text{ if $i = 1$} \\
    0 &\text{ if $i \ge 2$.}
\end{cases}
\end{equation*}
Therefore, by~\cite[Thm.~0.1]{HM16}, the induced map
\begin{equation*}
\bigl(\varphi|_{C(\mathbb{C}{\rm P}^n_q)}\bigr)_*:
K_0\bigl(C(\mathbb{C}{\rm P}^n_q)\bigr)\longrightarrow
K_0\bigl(C(\mathbb{C}{\rm P}^1_q)\bigr)
\end{equation*}
satisfies
\begin{equation*}
\bigl(\varphi|_{C(\mathbb{C}{\rm P}^n_q)}\bigr)_*
\Bigl([L^n_m]\Bigr)=[L^1_m]
\end{equation*}
for any $m\in\mathbb{Z}$. Since the index pairing computation of
\cite[Thm.~2.1]{Haj00} proves that $[L^1_m]=[L^1_k]$ implies
$m=k$ for $0<q<1$, and the $q=1$ case is well known, we thus arrive at the following: 
\begin{thm}
The spectral subspaces
$L^n_m$ are pairwise stably non-isomorphic as finitely generated projective
left modules over $C(\mathbb{C}{\rm P}^n_q)$. That is, for all
$n\in\mathbb{N}\setminus\{0\}$ and any $m,k \in \mathbb{Z}$, we have
\begin{equation*}
[L^n_m] = [L^n_k] \quad\mbox{if and only if}\quad m=k.
\end{equation*}
\end{thm}
\noindent
Observe that, for $0<q<1$,
 this fact was already proven in \cite[Prop. 4 and 5]{DL10} using an index pairing between the K-homology group
$K^0(C(\mathbb{C}{\rm P}^n_q))$  and the K-theory group $K_0(C(\mathbb{C}{\rm P}^n_q))$. 

To use the index pairing, for the rest of this section we assume $0<q<1$.
Denote by  $[\mu_0],\ldots,[\mu_n]\in K^0(C(\mathbb{C}{\rm P}^n_q))$ 
the K-homology generators constructed in \cite[Sec. 2]{DL10}. In the same paper, the authors proved that
\begin{equation}
\label{eq:pair_binom}
\langle[\mu_k], [L^n_{m}]\rangle = \binom{m}{k} \;
\end{equation}
for all $m \in\mathbb{N}$ and for all $0\leq k\leq n$. Here we adopt the convention that $\binom{m}{k}:=0$ when $k>m$. Furthermore, by {\cite[Prop. 3.2]{ABL15},
\begin{align}
\label{eq:pair_positive}
\langle [\mu_0], [L^n_{m}] \rangle = 1 \, \quad \text{and} \quad 
\langle [\mu_1], [L^n_{m}] \rangle = m \, 
\end{align}
for all $m\in\Z$ . We view the pairing with $[\mu_0]$ as computing the rank, and the pairing with $[\mu_1]$ as computing the noncommutative first Chern class. In agreement with the classical setting, we call $L^n_{-1}$ the section module of the noncommutative tautological line bundle, and we refer to $L^n_{1}$ as the section module of the dual noncommutative tautological line bundle. The latter is also known as the noncommutative Hopf line bundle. 

As mentioned in the introduction, 
the K-theory groups of quantum projective spaces have the same bases as their classical counterparts. 
This was proven in \cite{ABL15} by showing that, for $0\leq j \leq n$ and for $0\leq k\leq n$,
\begin{equation}
\label{eq:pair-un}
\inn{[\mu_k], t^j} = 
\begin{cases}
0 & \text{for} \quad j \not= k  \\
(-1)^j & \text{for} \quad j = k 
\end{cases} \, .
\end{equation}
Here $t:= [1] - [L^n_{1}]$ is the noncommutative Euler class of $L^n_{1}$.

We conclude this section by computing the pairings of the K-homology generators with the class of the section module of the noncommutative tautological line bundle $L^n_{-1}$.

\begin{prop}
\label{lem:pair}
For all positive integers $n$ and for all $0 \leq  k \leq n$, we have
$$
\langle [\mu_k], [L^n_{-1}] \rangle = (-1)^k.
$$
\end{prop}
\begin{proof}
 The cases $k=0,1$ are covered by \eqref{eq:pair_positive}.  For $2 \leq k \leq n$, we use the identity
\[\label{ttaut}
  [L^n_{-1}] = [L^n_{1}]^{-1} = (1-t)^{-1} = 1 + t + \cdots + t^n \quad \mbox{in} \ \  \mathbb{Z}[t]/t^{n+1}
\]
together with the pairings in \eqref{eq:pair-un}.
\end{proof}

\section{Vector bundles}

Although the $K_0$-groups of the Vaksman--Soibelman quantum projective spaces are determined by  associated  
noncommutative line bundles, expressing the generators as associated noncommutative vector bundles has its merits, as explained in the introduction.

For every $m=2, \dots, n$, consider the following surjection of Hopf algebras 
\[\label{pim}
\pi_m: \mathcal{O}(SU_q(m)) \longrightarrow \mathcal{O}(U(1)), \qquad \pi_m(U_{ij}) := \begin{cases} \delta_{ij}u^{-1} &  i<n \\
\delta_{ij} u^{m-1} & i=n \\
\end{cases} \, ,
\]
where $U_{ij}$ are the matrix coefficients of the fundamental representation of $SU_q(m)$.
To prove that $\pi_m$ is well defined, we have to verify that the determinant formulae (1.17) and (1.18) in \cite{Wor88} are satisfied. 
This can be done by a direct computation taking advantage of the fact that $u$ is unitary and that $\pi_m$ assigns zero to 
off-diagonal entries. With the help of $\pi_m$, we define the prolongations of principal comodule algebras:
$$
\mathcal{P}_{U(1)}(C(S^{2n+1}_q)) \overset{\mathcal{O}(U(1))}{\square} \mathcal{O}(SU_q(m)).
$$
Next, taking the fundamental representation $V_m$ of $SU_q(m)$, 
we construct the associated finitely generated projective left $C(\mathbb{C}{\rm P}^n_q)$-modules
\[\label{fnm}
\mathcal{F}_m^n := \mathcal{P}_{U(1)}(C(S^{2n+1}_q)) \overset{\mathcal{O}(U(1))}{\square} \mathcal{O}(SU_q(m)) \overset{\mathcal{O}(SU_q(m))}{\square} V_m.
\]


We are now ready to prove our main result:
\begin{thm}\label{main}
For any positive integer $n$ and $0<q\leq 1$, 
the classes 
$$
\cE^n_0:=[1],\quad \cE^n_1:= [L^n_{1}] -[1],\quad  \cE^n_m := [\mathcal{F}_m^n] - m[1] ,\quad m=2, \dots, n,
$$ 
form a basis of $K_0(C(\mathbb{C}{\rm P}^n_q))$.
\end{thm}
\begin{proof}
To begin with, plugging in the explicit formula \eqref{pim} into the definition \eqref{fnm}, we derive the decomposition
$$
\mathcal{F}_m^n =  ( L^n_{-1})^{\oplus (m-1)} \oplus L^n_{m-1}.
$$
Combining \eqref{atiyah} and \eqref{truncated}, we know that $\{1,t,\ldots,t^n\}$ is a basis of $K_0(C(\mathbb{C}{\rm P}^n_q))$.
We now use \eqref{ttaut} to write $\cE_m^n$ in the above basis
\[\label{basis}
\cE^n_0=1,\quad \cE^n_1=-t,\quad\cE_m^n=\sum_{k=2}^n\left((m-1)+(-1)^k\binom{m-1}{k}\right)t^k,\quad m\geq 2.
\]
\noindent 
It remains to show that the matrix ${\rm M}_n$ implementing \eqref{basis} is invertible over the integers.
Since $\mathrm{M}_{n+1}$ always contains $\mathrm{M}_n$ as an $n \times n$ 
submatrix in the upper-left corner, it is straightforward to verify the claim using elementary row operations and induction.
\end{proof}

For $n=2$, the above theorem yields 
$$
[1], \quad [L^2_{1}]-[1], \quad [L^2_{-1} \oplus L^2_{1}]-2[1]
$$
 as a basis of $K_0(C(\mathbb{C}{\rm P}^2_q))$. This agrees with the K-theory computations done in~\cite{FHMZ17} for the 
 quantum complex projective plane. Note also that the matrix ${\rm M}_n$ used in the above proof agrees up to a sign with
 the matrix of pairings $\inn{[\mu_k], \cE^n_j}$ that can be computed using \eqref{eq:pair_binom} and Proposition \ref{lem:pair}.

\section*{Acknowledgements}\noindent
This work is part of the project Quantum Dynamics partially supported by
EU-grant H2020-MSCA-RISE-2015-691246 and Polish Government grants
3542/H2020/2016/2 and  328941/PnH/2016. 
We acknowledge a substantial logistic support from the Max 
Planck Insitute for Mathematics in the Sciences in Leipzig, where much of this research was carried out. 
It is a pleasure to thank Nigel Higson for discussions and references.

\end{document}